\documentclass[12pt]{amsart}
\usepackage{amsmath,amssymb, amsfonts,amscd,graphicx,stmaryrd, mathabx}
\usepackage{epsfig}
\usepackage{epstopdf}
\usepackage{subfigure}
\usepackage{amsthm}
\usepackage{fullpage}
\usepackage{verbatim}
\usepackage{listings}
\usepackage{color}
\usepackage{mathtools}
\usepackage{multicol}
\usepackage{mathrsfs}
\usepackage{caption}

\newcommand{\lebn}

\usepackage[misc]{ifsym}
\usepackage{float}

\usepackage{subfig}

\numberwithin{equation}{section}

\theoremstyle{plain}

\newtheorem{thm}[equation]{Theorem}

\theoremstyle{definition}

\newtheorem{defn}[equation]{Definition}

\theoremstyle{plain}

\theoremstyle{definition}

\theoremstyle{remark}

\begin{document}

\title[Wolbachia]{A discussion of bisexual populations with Wolbachia infection as an evolution algebra}
\author{Songül ESİN, Müge KANUNİ, Barış ÖZDİNÇ}

\address{Department of Mathematics and Computer Science, İstanbul Kültür University, Ataköy
Kampüsü, Bakırköy 34158, İstanbul, Turkey.}
\email{s.esin@iku.edu.tr}
\address{Department of Mathematics, Düzce University,
Düzce, 81620, Turkey.}
\email{mugekanuni@duzce.edu.tr}
\address{CosmosID, Suite 300, 20030 Century Blvd, Germantown, 20874, MD, U.S.A.}
\email{baris@cosmosid.com}


\date{\today}
\thanks{The authors would like to thank Semiha Özgül for helpful discussion while conducting this work.}
\begin{abstract}
In this paper, Wolbachia infection in a bisexual and diploid population with a fixed cytoplasmic incompatibility rate $w$ and maternal transmission rate $d$  is studied as an evolution algebra.  As the cytoplasmic incompatibility (CI) of the population causes deaths in the offspring, the evolution algebra of this model is not baric, and is a dibaric algebra if and only if the cytoplasmic incompatibility rate $w$ is 1 and $d=1$. The idempotent elements are given in terms of $d$ and $w$. Moreover, this algebra has no absolute nilpotent elements when CI expression $w \neq 1$.
\end{abstract}
 
\maketitle

\section{Introduction}  \label{intro}
The manipulation of the host production by microbes is studied extensively in evolutionary biology and one particular type is a parasite called Wolbachia \cite{EcoEvo}. It is primarily found in insects and can be transferred to offspring, causing the mortality of the embryo of an infected male and an uninfected female. There are studies in the literature on Wolbachia-infected insects such as the terrestrial isopods \cite{isopod}, the honeybees \cite{WolbachBee};  and the mosquitoes \cite{Mos-Sing}. In Singapore, the release of male Wolbachia-infected mosquitoes reduced the dengue mosquitoes by causing mortality of the uninfected dengue mosquitoes and consequently reduced the dengue disease incidences. Wolbachia-infection has two different effects, vector-competency of dengue transmission is diminished by the strength of Wolbachia-infection; the second one is the appearance of cytoplasmic incompatibility in the population. \emph{Cytoplasmic incompatibility (CI)} is the reproductive incompatibility between males infected with a particular strain of bacteria and females not infected with this strain. 

There are two different mathematical approaches to the discussion of the Wolbachia populations. The first approach considers the infected population as a time-discrete dynamical system (\cite{I}). This article addresses the second approach namely, considering the Wolbachia infected population as an evolution algebra. The evolution algebra of bisexual populations is studied by Ladra and Rozikov in \cite{EABP}.  The Wolbachia-infected populations considered in this paper are bisexual and also diploid. Hence, Wolbachia-infection of bee populations is not considered in this paper as male bees are haploid. Moreover, there are different strains of Wolbachia bacteria apparent in biological systems, hence it might be interesting to study higher dimensional evolution algebras of Wolbachia-infected populations. However, within the scope of this paper, EABP is considered to have the minimum dimension, i.e. it is four-dimensional.      

The theory of evolution algebras dates back to Mendel and a concise approach to the topic is given in the expository article in \cite{history}. The mathematical theory of bisexual evolution algebras (EABP) is already established in \cite{EABP}, hence Wolbachia-infected population will be a particular example to implement the results known, and characterize its algebraic properties, list the similarities and the differences with \cite{EABP}.  
The Wolbachia-infected population considered in this article is a four-dimensional evolution algebra of a bisexual population denoted by $\mathcal{W}$. The basis consists of two types of females $f_1,f_2$ and two types of males $m_1, m_2$. Type 1 is a non-infected individual and type 2 is a Wolbachia-infected individual (denoted by superscript $+$), namely $f_1=XX, f_2=XX^+, m_1=XY, m_2=XY^+ $. 

The evolution algebra of a bisexual population with a Wolbachia infection,  $\mathcal{W}$, shares the same properties as an EABP (Theorem \ref{commutative}). However, the Wolbachia-infected population reproduction causes deaths in the offspring, hence there are some noteworthy differences.  The bisexual evolution algebra (EABP) is not a baric algebra, it is a dibaric algebra \cite[Theorems 5.1 and 6.3]{EABP}. On the other hand, $\mathcal{W}$ is not a baric algebra (Theorem \ref{baric}), and it is a dibaric algebra if and only if $w=1=d$(Theorem \ref{dibaric}).  

The outlay of the paper will be as follows: Section \ref{section:prel} starts with the preliminaries from biology and evolution algebras. Section \ref{section:EABP} mimics the known results on a bisexual evolution algebra in \cite{EABP} to the case of a Wolbachia-infected bisexual population. Bustamante, Mellon and Velasco in \cite{Matrix}, propose a method to determine whether a genetic algebra is an evolution algebra. Hence, using this method, it is shown that the Wolbachia-infected bisexual population is not an evolution algebra in Section \ref{matrix}. Section \ref{section:evolop} studies the idempotent and absolute nilpotent elements of $\mathcal{W}$. Theorem \ref{idempotent} lists the idempotent elements and agrees with the results in \cite[Section 3]{I}. Moreover, Theorem \ref{nilpotent} shows that there are no absolute nilpotent elements when CI rate is not 1. 

\section{Preliminaries}\label{section:prel}
\subsection{Biology}
\begin{defn}\cite{EcoEvo}\label{defn:CI}
 \emph{Cytoplasmic incompatibility (CI)} is the reproductive incompatibility
between males infected with a particular strain of bacteria and females not infected with this strain. 
\end{defn}
Wolbachia is a particular example of a parasite in insects that is transmitted via reproduction. Although there are different strains of Wolbachia-infections in nature, to minimize the dimension of evolution algebra studied in this paper, only one strain of Wolbachia-infection is considered.

The cytoplasmic incompatibility of the population is given as $w$, and $d$ is the probability of the transmission of the Wolbachia infection from a female to its offspring. In biological models, $w$ is also known as the paternal affection rate and $d$ is the maternal transmission rate (eg.\cite{Fine}). If both $w=d=1$, the infected individual (male/female) produces all infected alleles. Mathematically speaking, if $w=0$, then the infected male produces no infected alleles and if $d=0$, then the infected female does not produce any infected gametes. Both $w=0$ or $d=0$ scenarios are not biologically observed, and not interesting. Hence, throughout the text, assume $w,d \in (0,1]$. 

The infected female/male produces gametes that are both infected and uninfected. The zygote from an infected male gamete and an uninfected female gamete is viable, hence there are deaths in the offspring population.   

\begin{table}[h]
\caption{Gamete Crossing in a Wolbachia-infected Population.}
\[
\scalebox{0.8}{
\begin{tabular}{|c|c|c|c|c|}
\hline
\multicolumn{1}{|c|}{{\small Gamete Crossing}}  & $X$ & $X^{+}$ & $Y$ & $Y^{+}$ \\ \hline
\multicolumn{1}{|c|}{$
\begin{array}{c}
\text{ \ } \\ 
\end{array} X \begin{array}{c}
\text{ \ } \\ 
\end{array} $} & uninfected female & (death) no offspring & uninfected male
& (death) no offspring \\ \hline
\multicolumn{1}{|c|}{$ \begin{array}{c}
\text{ \ } \\ 
\end{array} X^{+} \begin{array}{c}
\text{ \ } \\ 
\end{array}  $} & infected female & infected female & infected male & 
infected male\\ \hline
\end{tabular} } 
\]
\end{table}

\subsection{Evolution Algebra} In this subsection, the necessary terminology, and historical remarks are quoted from \cite{history}, which is a brief and concise summary of the history of evolution algebras. 

Interpretation of sexual reproduction laws of inheritance  with algebraic symbols dates back to Mendel. More precise studies by Serebrowski, Kostitzin, Glivenkov, and Etherington (to name a few pioneers of the subject) gave rise to the term genetic algebras. Algebraic properties of special genetic algebras are studied by several authors. 

Etherington also defined baric algebra in 1939 as a special genetic algebra. 
\begin{defn} \label{Defn:baric}
A \textit{character} for an algebra $A$ is a nonzero multiplicative linear form on $A$, that is, a nonzero algebra homomorphism from $A$ to $\mathbb{R}$. Not every algebra admits a character. For example, an algebra with zero multiplication has no character. A pair $(A,\sigma )$ consisting of an algebra $A$ and a character $\sigma $ on $A$ is called a \textit{baric algebra}. 
\end{defn}

In 1970, Holgate introduced the notions of sex differentiation algebra (Definition \ref{Defn:sex}) and dibaric algebra (Definition \ref{Defn:dibaric}). 

\begin{defn} \label{Defn:sex} \cite[Definition 6.1]{EABP}
Let $\mathcal{U}=\left\langle W,M\right\rangle _{\mathbb{R}}$ denote a two-dimensional commutative algebra over 
$\mathbb{R}$ with the multiplication table $$ W^{2}=M^{2}=0,~~~~~\ WM=\frac{1}{2}(W+M). $$
Then $\mathcal{U}$ is called the \textit{sex differentiation algebra.}
\end{defn}
Notice that $\mathcal{U}^2 = span \{zt \ | \ z,t \in \mathcal{U} \} = \left\langle W+M\right\rangle _{\mathbb{R}}$ is an ideal of $\mathcal{U}$. 
\begin{defn} \label{Defn:dibaric} 
 If an $\mathbb{R}$-algebra $A$ admits a homomorphism onto the sex differentiation algebra then $A$ is called a \textit{dibaric algebra.}
\end{defn}
Holgate also proved that if $A$ is a dibaric algebra, then $A^2$ is a baric algebra. 

The evolution algebras were introduced by Tian in his Ph.D. thesis to model the self-reproduction rules of non-Mendelian genetics in 2004.
\begin{defn} \label{TianEvolAlg} 
Let $I$ be an index set and  $E$ be a vector space over a field $\mathbb K$, with a basis $B=\{e_{i} \ \vert \ i\in I \}$ such that $e_{i}e_{j}=0$ whenever $i\neq j$ and $e_ie_i=\sum_{k\in I} \omega _{ki}e_{k}$. Then $E$ is called an \emph{evolution algebra} over $\mathbb K$ and $B$ is a natural basis of $E$.  
 The scalars $\omega_{ki}\in \mathbb K$ are the \emph{structure constants} of $A$\emph{\ relative to} $B$,  the matrix $M_B:= \left(w_{ki}\right)$ is the \emph{structure matrix of} $A$ \emph{relative to} $B$. Every evolution algebra is uniquely determined by its structure matrix. 
\end{defn}

Assume that a population consists of $m$ different genetic types and consider the $m$-tuple $x=(x_{1},\ldots ,x_{m})$. Each component $x_i$ of $x$ denotes the probability that a random individual in the population belongs to the species that is determined by $i^{th}$ genetic type, hence $x_{i}\geq 0$ and $\sum\limits_{i=1}^{m}x_{i}=1$. 
Let $x^{0}=(x_{1}^{0},\ldots ,x_{m}^{0})$ be the
probability distribution of species in the initial generations, and  $P_{ij,k}$ be the probability that individuals in the $i^{th}$ and $j^{th}$ species interbreed to produce an individual $k$. Then the probability distribution 
$x^{\prime }=(x_{1}^{\prime },\ldots ,x_{m}^{\prime })$ of the species in the first generation can be found by the total probability i.e. $$ x_{k}^{\prime }=\sum\limits_{i,j=1}^{m}P_{ij,k}x_{i}^{0}x_{j}^{0}, \quad \quad k=1,\ldots ,m $$
where the cubic matrix $P =(P_{ij,k})_{i,j,k=1}^{m}$ satisfies the
following conditions
$$P_{ij,k}\geq 0,~~\sum\limits_{k=1}^{m}P_{ij,k}=1,~~~~~i,j\in \{1,\ldots ,m\}.$$
Bernstein defined the term \emph{quadratic stochastic operator} (QSO) as a map $V : S^{m-1} \rightarrow S^{m-1}$ where 
$$ S^{m-1}=\left\{ (x_{1},\ldots ,x_{m})\in 
\mathbb{R}^{m}~|~x_{i}\geq 0,~~\sum\limits_{i=1}^{m}x_{i}=1\right\}$$ 
with $x^{0}\longmapsto x^{\prime }$. This map 
$V$ is the evolutionary operator that describes the inheritance process of a free population with 
$m$ different genetic types.
The evolution algebra of a bisexual population (EABP) is described in \cite{EABP}. When the population is bisexual, the basis is partitioned into a set of females with different types indexed by $\{1, 2,...,n \}$, and the set of male types indexed by $\{1, 2,..., \upsilon \}$. The dimension of the population is the sum of the male and female types, 
that is $n+\upsilon$. The population is described by its state vector $(x, y) \in S^{n-1} \times S^{\upsilon -1}$, the product of two unit simplexes in $\mathbb{R}^n$ and $\mathbb{R}^{\upsilon}$ respectively. Vectors $x$ and $y$ are the probability distributions of the females and males over the possible types which satisfy the equations: 
$$ x_i \geq 0, y_j \geq 0,  \mbox{ for all } i \in \{1, 2,...,n \}, j \in \{1, 2,...,\upsilon \}$$
$$ \sum_{i=1}^n x_i =1  \mbox{ and }  \sum_{i=1}^{\upsilon} y_i =1.$$

Let $P_{ik,j}^{(f)}$ and $P_{ik,l}^{(m)}$ be inheritance coefficients defined as the probability that female offspring is type $j$ and, respectively, that a male offspring is of type $l$ when the parental pair
is $ik$ ($i,j=1,\ldots,n$; and $k,l=1,\ldots,\upsilon $). We have
\begin{equation}\label{sumprobcoef=1}
P_{ik,j}^{(f)}\geq 0,~~~~\sum\limits_{j=1}^{n}P_{ik,j}^{(f)}=1;\text{ \ \ \
\ \ }P_{ik,l}^{(m)}\geq 0,~~~~\sum\limits_{l=1}^{\upsilon }P_{ik,l}^{(m)}=1
\end{equation}

 \section{Evolution algebra model of a Wolbachia infected population}\label{section:EABP} 
The study of the Wolbachia-infected population as an evolution algebra is a particular example of the bisexual population. Wolbachia-infected populations considered in this paper do not form a model for an evolution algebra in the sense of Definition \ref{TianEvolAlg} as shown in Theorem \ref{Notevalg}. However, the evolution algebra of a bisexual population \cite{EABP} is a suitable model.  Most of this work is an interpretation of this problem with respect to the paper of Ladra and Rozikov \cite{EABP}.  

Consider a population with four types of individuals, males without Wolbachia infection: $XY$, males with Wolbachia infection: $XY^{+}$, females without Wolbachia infection: $XX$ and females with Wolbachia infection: $XX^{+}$ where the basis elements of the population are denoted by $f_1=XX,~f_2=XX^{+},~m_1=XY,~m_2=XY^{+}$.  

Define $\mathcal{W}$ as the vector space generated by four basis elements $\mathcal{B}=\{f_1, f_2, m_1, m_2 \}$ with the maternal transmission rate is given as $d$ and the cytoplasmic incompatibility of the population \emph{CI} rate is given as $w$. Stated differently, the type 2 (Wolbachia-infected) female individual ($f_2$) produces infected gamete with probability $d$ and type 2 (Wolbachia infected) male individual ($m_2$) produces infected gamete with probability $w$ ($w$ is also called the paternal affection rate in literature \cite{Fine}). 
An individual infected female produces an  infected gamete $X^+$ with probability $d$, and an uninfected gamete $X$ with probability $1-d$; whereas an individual infected male produces infected gametes $X^+$ or $Y^+$ with probability $w/2$, and uninfected gametes $X$ or $Y$ with probability $(1-w)/2$ each. Hence, the zygotes formed contribute to the offspring population. The Punnett squares of gametes of the basis elements are given in Table \ref{tab:punnett}.

\begin{table}[h]
\caption{Punnett squares of gametes from the mating of basis elements \\ \small
    (a) $XX$ vs. $XY$, (b) $XX$ vs. $XY^+$, (c) $XX^+$ vs. $XY$ (d) $XX^+$ vs. $XY^+$.}
\label{tab:punnett}
\vspace{.5cm}
    \begin{subtable}[]
    \centering
    \hspace{.4cm}
    \scalebox{0.95}{
        $ \begin{tabular}{c|c|c}
& $X$ & $Y$ \\ \hline\hline
&   &   \\ 
$X$ & $\frac{1}{2}XX$ & $\frac{1}{2}XY$
        \end{tabular} $ }
    \end{subtable}
   \hfill
    \begin{subtable}[] 
    \centering 
    \hspace{0.7cm}
    \scalebox{0.9}{ $\begin{tabular}{c|c|c|c|c}
& $X$ & $X^{+}$ & $Y$ & $Y^{+}$ \\ \hline\hline
&   &   &   &   \\ 
$X$ & $\frac{1}{2}(1-w)XX$ & $-$ & $\frac{1}{2}(1-w)XY$ & $-$%
        \end{tabular}$ 
        }
    \end{subtable}
\bigbreak
\vspace{.5cm}
    \begin{subtable}[]
       \centering
         \hspace{.4cm}
    \scalebox{0.7}{
        $\begin{tabular}{c|c|c}
& $X$ & $Y$ \\ \hline\hline
&   &   \\ 
$X$ & $\frac{1}{2}(1-d)XX$ & $\frac{1}{2}(1-d)XY$ \\ 
&   &   \\ \hline
&   &   \\ 
$X^{+}$ & $\frac{1}{2}dXX^{+}$ & $\frac{1}{2}dXY^{+}$%
        \end{tabular}$}
    \end{subtable}
    \hfill
    \begin{subtable}[] 
    \centering
     \hspace{0.7cm}
        \scalebox{0.6}{ $\begin{tabular}{c|c|c|c|c}
& $X$ & $X^{+}$ & $Y$ & $Y^{+}$ \\ \hline\hline
&   &   &   &   \\ 
$X$ & $\frac{1}{2}(1-d)(1-w)XX$ & $-$ & $\frac{1}{2}(1-d)(1-w)XY$ & $-$ \\ 
&   &   &   &   \\ \hline
&   &   &   &   \\ 
$X^{+}$ & $\frac{1}{2}d(1-w)XX^{+}$ & $\frac{1}{2}dwXX^{+}$ & $\frac{1}{2}d(1-w)XY^{+}$ & $\frac{1}{2}dwXY^{+}$
        \end{tabular}$ 
         }
    \end{subtable} 
    \end{table}

Summing up the information from the biological model in Table \ref{tab:punnett}, the multiplication table of $\mathcal{W}$ is achieved in Table \ref{tab:zygote}.
\begin{table}[h]
\caption{The multiplication table of $\mathcal{W}$.} 
\label{tab:zygote} 
\[
\begin{tabular}{c|c|c}
& $XY$ & $XY^{+}$ \\ \hline\hline
&   &   \\ 
$XX$ & $\frac{1}{2}XX+\frac{1}{2}XY$ & $\frac{1}{2}(1-w)XX+\frac{1}{2}(1-w)XY
$ \\ 
&   &   \\ \hline
&   &   \\ 
$XX^{+}$ & $
\begin{array}{c}
\frac{1}{2}(1-d)XX+\frac{1}{2}dXX^{+} \\ 
\text{ \ } \\ 
+\frac{1}{2}(1-d)XY+\frac{1}{2}dXY^{+}
\end{array}
$ & $
\begin{array}{c}
\frac{1}{2}(1-d)(1-w)XX+\frac{1}{2}dXX^{+} \\ 
\text{ \ } \\ 
+\frac{1}{2}(1-d)(1-w)XY+\frac{1}{2}dXY^{+}
\end{array}
$
\end{tabular} 
\]
\end{table}

In $\mathcal{W}$, whether crossed with a type 1 or type 2 male, a type 1 female will never produce a type 2 male or female. When a type 1 female and type 1 male mate the offspring will only be a type 1. Type 1 female crossing with type 2 male will produce type 1 female (male) with probability $1-w$ (See Table \ref{tab:zygote}).

Also, let $P_{ik,j}^{(f)}$ and $P_{ik,j}^{(m)}$ be inheritance coefficients defined as the probability that an offspring ( $f$ for female, $m$ for male) is type $j$, ($j=1$ for non-infected and $j=2$ for infected) when the parental pair is $ik$ ($i$ denotes the mother's and $k$ denotes the father's types respectively). 
Table \ref{tab:zygote} states the frequency of the individuals within the whole population, but to mimic the model in \cite{EABP}, multiply each coefficient by 2 in Table \ref{tab:zygote} to get the inheritance coefficients of females as the frequency of the female basis element to the total female population (Table \ref{tab:coef}). 

\begin{table}[h]
\caption{The inheritance coefficients $ P^{(f)}_{ik,j}$ of $\mathcal{W}$.} 
\label{tab:coef} 
\[
\begin{tabular}{|c|c|c|c|}
\hline
&  &  &  \\ 
$P_{11,1}^{(f)}=1$ & $P_{12,1}^{(f)}=1-w$ & $P_{21,1}^{(f)}=1-d$ & $%
P_{22,1}^{(f)}=(1-d)(1-w)$ \\ 
&   &   &   \\ 
\hline 
&  &  &  \\ 
$P_{11,2}^{(f)}=0$ & $P_{12,2}^{(f)}=0$ & $P_{21,2}^{(f)}=d$ & $P_{22,2}^{(f)}=d$ \\ 
&  &  &  \\ \hline
\end{tabular}%
\]
\end{table}

Similarly, the inheritance coefficients of the male individuals are the same, so for all $i,j,k \in \{ 1,2\}$, $P_{ik,j}^{(f)} = P_{ik,j}^{(m)}$. 
Notice that as there are deaths in the offspring of the mating of $f_1$ with $m_2$ and $f_2$ with $m_2$, the sum of the inheritance coefficients in the second and fourth column respectively do not add up to 1.  
\bigskip

Now, the multiplication is defined on the basis $\mathcal{B}$ as  
\begin{eqnarray*}
f_i m_{k}&=&m_{k}f_{i}=\frac{1}{2}\left(
P_{ik,1}^{(f)}f_{1}+P_{ik,2}^{(f)}f_{2}+P_{ik,1}^{(m)}m_{1}+P_{ik,2}^{(m)}m_{2}\right),
\\
f_{i}f_{j} &=&0,\text{ \ }i,j=1,2;\text{ \ \ \ \ \ }%
m_{k}m_{l}=0,\text{ \ }k,l=1,2.
\end{eqnarray*}

Hence, 
\begin{eqnarray} \label{3.1}
f_1m_1 &=&\dfrac{1}{2}\left(
P_{11,1}^{(f)}f_1+P_{11,2}^{(f)}f_2+P_{11,1}^{(m)}m_1+P_{11,2}^{(m)}m_2\right) \nonumber
\\
&=&\dfrac{1}{2}\left( f_1+m_1\right) \nonumber  \\
f_1m_2 &=&\dfrac{1}{2}\left(
P_{12,1}^{(f)}f_1+P_{12,2}^{(f)}f_2+P_{12,1}^{(m)}m_1+P_{12,2}^{(m)}m_2\right) \nonumber 
\\
&=&\dfrac{1}{2}(1-w)\left( f_1+m_1\right) \nonumber  \\
f_2m_1 &=&\dfrac{1}{2}\left(
P_{21,1}^{(f)}f_1+P_{21,2}^{(f)}f_2+P_{21,1}^{(m)}m_1+P_{21,2}^{(m)}m_2\right) 
\\
&=&\dfrac{1}{2}\left[ (1-d)\left( f_1+m_1\right) +d\left(
f_2+m_2\right) \right]  \nonumber \\
f_2m_2 &=&\dfrac{1}{2}\left(
P_{22,1}^{(f)}f_1+P_{22,2}^{(f)}f_2+P_{22,1}^{(m)}m_1+P_{22,2}^{(m)}m_2\right) \nonumber 
\\
&=&\dfrac{1}{2}\left[ (1-d)(1-w)\left( f_1+m_1\right)
+d\left( f_2+m_2\right) \right] \nonumber 
\end{eqnarray}
\begin{defn}
    The algebra $\mathcal{W}$ generated by $\mathcal{B}=\{f_1, f_2, m_1, m_2 \}$ for a given \emph{CI} (paternal affection rate) $w$ and maternal transmission rate $d$, with the multiplication in Equation (\ref{3.1}) is called \emph{the evolution algebra of the Wolbachia-infected bisexual population}.
\end{defn}
The population is described by its state vector $(x, y) \in S^{1} \times S^{1}$, the product of two unit simplexes in $\mathbb{R}^2$. Vectors $x=(x_1,x_2)$ and $y=(y_1,y_2)$ are the probability distributions of the females and males over the possible types which satisfy the equations: 
$$ x_i \geq 0, y_j \geq 0,  \mbox{ for all } i,j \in \{1, 2 \}  \mbox{ and }
 \sum_{i=1}^2 x_i =1  = \sum_{i=1}^2 y_i .$$
In terms of inheritance coefficients, 
$$
P_{ik,j}^{(f)}\geq 0, \ P_{ik,l}^{(m)}\geq 0, ~~~~\sum\limits_{j=1}^{2}P_{i1,j}^{(f)}=1= \sum\limits_{j=1}^{2}P_{i1,j}^{(m)}; 
$$
However,
\begin{equation}\label{sumprobcoefnot1}
\sum\limits_{j=1}^{2}P_{12,j}^{(f)}= 1-w =\sum\limits_{j=1}^{2}P_{12,j}^{(m)} \text{ and } \sum\limits_{j=1}^{2}P_{22,j}^{(f)}= 1-w + dw =\sum\limits_{j=1}^{2}P_{22,j}^{(m)} 
\end{equation}

\bigskip

\bigskip 

Notice that, Wolbachia-infected bisexual population $\mathcal{W}$ differs from EABP of \cite{EABP} as listed:
\begin{itemize}
    \item  The first observation is that, unlike in an EABP, 
    in $\mathcal{W}$ when $k=2$, as Equation (\ref{sumprobcoefnot1}) states
\begin{equation*}
 \sum\limits_{j=1}^{2}P_{12,j}^{(f)}= 1-w \neq 1  \quad  \mbox{ and } \quad
 \sum\limits_{j=1}^{2}P_{22,j}^{(f)}= 1-w +dw \neq 1 
\end{equation*}
    \item  In \cite[Remark 3.1]{EABP}, it states that
if a population is free then the male and female types are identical and, in particular number of elements in the female basis is equal to the number of elements in the male basis  (i.e. $n=\upsilon=2$), the inheritance coefficients are the same for male
and female offspring, that is
\begin{equation}\label{inhercoefeq}
P_{ik,j}=P_{ik,j}^{(f)}=P_{ik,j}^{(m)}.
\end{equation}
Although Equation (\ref{inhercoefeq}) is satisfied in $\mathcal{W}$, male and female types are not identical. For instance,   whether crossed with a type 1 or type 2 male, a type 1 female will never produce a type 2 male or female. On the other hand, a type 1 male crossed with type 2 female will produce an offspring of type 2 male or female with probability $d$ (See Table \ref{tab:zygote}). Hence, male and female types are not identical and $\mathcal{W}$ is not a free population.
 
    \item Moreover, in $\mathcal{W}$, the symmetry condition $P_{ik,j}=P_{ki,j}$ is not necessarily satisfied, as $P_{12,1} = 1 -w \neq 1 - d =  P_{21,1}$ and $P_{21,2}= d  \neq 0 = P_{12,2}$.
\end{itemize}

\section{Wolbachia-infected populations do not form a model for an evolution algebra} \label{matrix} 

Consider  $\mathcal{W}$ with $\mathcal{B}=\{f_1, f_2, m_1, m_2 \}$ for a given \emph{CI} rate of $w$ and maternal transmission rate $d$. 

In a recent work of Bustamante, Mellon, and Velasco, the authors analyze when a genetic algebra is an evolution algebra by determining whether the structure matrices of the genetic algebra are simultaneously diagonalizable (\cite[Theorems 5,6]{Matrix}). Following the same notation, define  $\pi_k : \mathcal{W} \rightarrow \mathbb{R}$ as $\pi_k(f_if_j)= 0 = \pi_k(m_im_j)$, $\pi_k(f_im_j)= \frac{1}{2} P_{ij,k}$ for $k \in \{1,2\}$. Equation (\ref{inhercoefeq}), that is $P^{(f)}_{ij,k} = P_{ij,k} = P^{(m)}_{ij,k}$; reduces the four structural matrices to two different matrices. 

By using Table \ref{tab:coef}, two distinct structural matrices are derived:  
\smallskip

$$M_{1}(\mathcal{B})=\left[ 
\begin{array}{cccc}
\pi _{1}(f_{1}f_{1}) & \pi _{1}(f_{1}f_{2}) & \pi _{1}(f_{1}m_{1}) & \pi
_{1}(f_{1}m_{2}) \\ 
\pi _{1}(f_{2}f_{1}) & \pi _{1}(f_{2}f_{2}) & \pi _{1}(f_{2}m_{1}) & \pi
_{1}(f_{2}m_{2}) \\ 
\pi _{1}(m_{1}f_{1}) & \pi _{1}(m_{1}f_{2}) & \pi _{1}(m_{1}m_{1}) & \pi
_{1}(m_{1}m_{2}) \\ 
\pi _{1}(m_{2}f_{1}) & \pi _{1}(m_{2}f_{2}) & \pi _{1}(m_{2}m_{1}) & \pi
_{1}(m_{2}m_{2})
\end{array}
\right] $$

$$ =\left[ 
\begin{array}{cccc}
0 & 0 & \frac{1}{2} & \frac{1-w}{2} \\ 
&  &  & \text{ \ } \\ 
0 & 0 & \frac{1-d}{2} & \frac{(1-d)(1-w)}{2} \\ 
&  &  & \text{ \ } \\ 
\frac{1}{2} & \frac{1-w}{2} & 0 & 0 \\ 
&  &  & \text{ \ } \\ 
\frac{1-d}{2} & \frac{(1-d)(1-w)}{2} & 0 & 0%
\end{array}%
\right] $$ 

and 
$$M_{2}(\mathcal{B})=\left[ 
\begin{array}{cccc}
\pi _{2}(f_{1}f_{1}) & \pi _{2}(f_{1}f_{2}) & \pi _{2}(f_{1}m_{1}) & \pi
_{2}(f_{1}m_{2}) \\ 
\pi _{2}(f_{2}f_{1}) & \pi _{2}(f_{2}f_{2}) & \pi _{2}(f_{2}m_{1}) & \pi
_{2}(f_{2}m_{2}) \\ 
\pi _{2}(m_{1}f_{1}) & \pi _{2}(m_{1}f_{2}) & \pi _{2}(m_{1}m_{1}) & \pi
_{2}(m_{1}m_{2}) \\ 
\pi _{2}(m_{2}f_{1}) & \pi _{2}(m_{2}f_{2}) & \pi _{2}(m_{2}m_{1}) & \pi
_{2}(m_{2}m_{2})%
\end{array}%
\right] =\left[ 
\begin{array}{cccc}
0 & 0 & 0 & 0 \\ 
0 & 0 & \frac{d}{2} & \frac{d}{2} \\ 
0 & 0 & 0 & 0 \\ 
\frac{d}{2} & \frac{d}{2} & 0 & 0%
\end{array}%
\right] .$$

Note that $M_2(\mathcal{B})$ is not a symmetric matrix. 

\begin{thm} \label{Notevalg}
    $\mathcal{W}$ is not an evolution algebra in the sense of Definition \ref{TianEvolAlg}. 
\end{thm}
\begin{proof}
Recall the fact from \cite[Theorem 1.3.12]{MatrixAnalysis}: Assume $A,B$ are $n \times n$ diagonalizable matrices. Then $A,B$ are simultaneously diagonalizable if and only if $A,B$ commute. 

Now, $M_1(\mathcal{B})$,  $M_2(\mathcal{B})$ are both diagonalizable with the diagonal matrices $D_1$ and $D_2$:
\begin{eqnarray*}
D_1= \left[ 
\begin{array}{cccc}
\frac{(1-d)(1-w)+1}{2} & 0 & 0 & 0 \\ 
&  &  & \text{ \ } \\ 
0 & \frac{(1-d)(1-w)-1}{2}  & 0 & 0 \\
&  &  & \text{ \ } \\ 
0 & 0 & 0 & 0 \\ 
&  &  & \text{ \ } \\ 
0 & 0 & 0 & 0 \end{array} \right] 
\quad \quad 
D_2 = \left[ 
\begin{array}{cccc}
0 & 0 & 0 & 0 \\
&  &  & \text{ \ } \\ 
0 & 0 & 0 & 0 \\ 
&  &  & \text{ \ } \\ 
0 & 0 & -\frac{d}{2} & 0 \\ 
&  &  & \text{ \ } \\ 
0 & 0 & 0 & \frac{d}{2} \end{array} \right]
\end{eqnarray*}
respectively. However,    
$$M_{1}(\mathcal{B})M_{2}(\mathcal{B})= \left[ 
\begin{array}{cccc}
-\frac{d(w-1)}{4} & -\frac{d(w-1)}{4} & 0 & 0 \\
&  &  & \text{ \ } \\
\frac{d(d-1)(w-1)}{4} & \frac{d(d-1)(w-1)}{4} & 0
& 0 \\ 
&  &  & \text{ \ } \\
0 & 0 & -\frac{d(w-1)}{4} & -\frac{d(w-1)}{4} \\
&  &  & \text{ \ } \\
0 & 0 & \frac{d(d-1)(w-1)}{4} & \frac{d(d-1)(w-1)}{4}%
\end{array}
\right]$$

$$M_{2}(\mathcal{B})M_{1}(\mathcal{B}) = \left[ 
\begin{array}{cccc}
0 & 0 & 0 & 0 \\ 
&  &  & \text{ \ } \\
-\frac{d(d-2)}{4} & \frac{d(d-2)(w-1)}{4} & 0 & 0 \\ 
&  &  & \text{ \ } \\
0 & 0 & 0 & 0 \\
&  &  & \text{ \ } \\
0 & 0 & -\frac{d(d-2)}{4} & \frac{d(d-2)(w-1)}{4}
\end{array}
\right]$$
That is, $M_1(\mathcal{B}), M_2(\mathcal{B})$ do not commute. Hence, by \cite[Theorem 1.3.12]{MatrixAnalysis} are not simultaneously diagonalizable. Therefore, $\mathcal{W}$ is not an evolution algebra by \cite{Matrix}. 
\end{proof}

As stated in \cite[Theorem 4.1]{EABP} for an evolution algebra of a bisexual population, the following result is valid for an evolution algebra of a Wolbachia-infected bisexual population $\mathcal{W}$. Recall the definition of flexible and power-associative algebra.
\begin{defn}
 An algebra $A$ is called {\em flexible} if $z(tz)=(zt)z$ for any $z,t \in A$. An algebra $A$ is {\em power-associative} if $(zz)(zz)=((zz)z)z=(z(zz))z$ for every $z$ of $A$.
\end{defn}

\begin{thm} \label{commutative} Let $\mathcal{W}$ be the evolution algebra of a bisexual population with a Wolbachia infection.
\begin{itemize}
\item[(1)] $\mathcal{W}$ is not necessarily associative.

\item[(2)] $\mathcal{W}$ is commutative and flexible.

\item[(3)] $\mathcal{W}$ is not necessarily power-associative.
\end{itemize}
\end{thm}

\begin{proof}
(1) Take $f_1,m_1$ with $P_{11,1}^{(f)}=1\neq 0$ and
take $m_2$ with $P_{12,1}^{(f)}=1-p\neq 0$ then 
\begin{equation*}
\left( f_1m_1\right) m_2=\dfrac{1}{2}\left(
f_1+m_1\right) m_2=\dfrac{1}{2}%
f_1m_2=\dfrac{1}{4}(1-w)\left(
f_1+m_1\right)\neq 0 
\end{equation*} 

But $f_1(m_1m_2)=0,$ i.e. $\left(
f_1m_1\right) m_2\neq
f_1(m_1m_2).$
\begin{itemize}
    
\item[(2)] 
It is clear that for any $z,t\in \mathbb{R}^{2\times 2}$ we have
\begin{eqnarray*}
z
&=&(x,y)=x_{1}f_1+x_{2}f_2+y_{1}m_1+y_{2}m_2,
\\
t
&=&(u,v)=u_{1}f_1+u_{2}f_2+v_{1}m_1+v_{2}m_2.
\end{eqnarray*}
As $f_im_k = m_kf_i$ $(k,l=1,2)$, $zt=tz.$ \\
Since $\mathcal{W}$ is commutative, it follows that $(zt)z=z(tz)$ is true. Hence it is flexible.
\end{itemize}
\begin{itemize}
\item[(3)] To show that $\mathcal{W}$ is not power-associative,
we will construct an example of $z$ such that $(zz)(zz)\neq ((zz)z)z.$ \\
Consider $z=f_1+m_2.$ Then 
\begin{equation*}
z^{2}=2f_1m_2=(1-w)\left( f_1+m_1\right) 
\end{equation*}
and 
\begin{equation*}
z^{2}z^{2}=2(1-w)^2f_1m_1=(1-w)^2\left(
f_1+m_1\right) .
\end{equation*}
On the other hand, 
\begin{eqnarray*}
z^{2}z &=&(1-w)\left( f_1m_1+f_1m_2\right)\\
&=&\dfrac{1}{2}(1-w)\left[ \left( f_1+m_1\right) +(1-w)^2\left(
f_1+m_1\right) \right]  \\
&=&\dfrac{1}{2}(1-w)(2-w)\left( f_1+m_1\right) 
\end{eqnarray*}
and 
\begin{eqnarray*}
\left( z^{2}z\right) z &=&\dfrac{1}{2}(1-w)(2-w)\left(
f_1m_1+f_1m_2\right)  \\
&=&\dfrac{1}{4}(1-w)(2-w)^{2}\left( f_1+m_1\right) .
\end{eqnarray*}
This shows that $(zz)(zz)\neq ((zz)z)z.$
\end{itemize}
\end{proof}
\subsection{$\mathcal{W}$ is not a baric algebra} 
Recall that an $\mathbb{R}$-algebra $A$ is a baric algebra if it admits a nonzero algebra map $\sigma : A \rightarrow \mathbb{R}$.
\begin{thm}\label{baric}
$\mathcal{W}$ is not a baric algebra.
\end{thm}

\begin{proof}
Consider a character $\sigma : \mathcal{W} \rightarrow \mathbb{R} $ such that $\sigma (f_i)=a_{i},\sigma
(m_j)=b_{j}$ for $i,j=1,2.$ Now, 
\begin{eqnarray*}
\sigma (f_if_i) &=&\sigma (f_i)\sigma
(f_i)=a_{i}a_{i} \\
\sigma (0) &=&0=a_{i}a_{i}\Rightarrow a_{1}=a_{2}=0.
\end{eqnarray*}%
In a similar manner, any basis vector should be mapped to $0.$ Hence, $
\sigma \equiv 0$, $\mathcal{W}$ does not have a nonzero character map. 
\end{proof}

\subsection{$\mathcal{W}$ is a dibaric algebra when $w=1=d$}

 \begin{thm}\label{dibaric} $\mathcal{W}$ is a dibaric algebra if and only if $w=1=d$.
\end{thm}

\begin{proof} Assume $\mathcal{U}$ is the sex differentiation algebra generated by $W$ and $M$. If $\mathcal{W}$ is a dibaric algebra, then there is an onto homomorphism $\varphi :\mathcal{W}\rightarrow \mathcal{U}$. 

Assume, $\varphi (f_{i})=a_{i}W + a_{i}'M$  and $\varphi (m_{j})=b_{j}W + b_{j}'M $ for $i,j=1,2$,\\
Then $\varphi (f_{i}f_i)=\varphi(0)= 0 = (a _{i}W + a_{i}'M)(a_{i}W + a_{i}'M) = (a_{i} a_{i}')(W + M)$ \\ implies $a_{i} a_{i}' =0$, either $a_{i} =0 $ or $a_{i}'=0$. 
A similar argument will show that either $b_{j} =0 $ and or $b_{j}'=0$. 
  
{\bf Claim:} When the images of $f_i$ are both non-zero, then either both $f_1,f_2$ are mapped to multiples of $W$ or to multiples of $M$. A similar claim holds for $m_j$.  

Assume on the contrary that $\varphi (f_1)= a_{1}W $ and $\varphi (f_2)= a_{2}'M$. 
Then $\varphi (f_1f_2)=\varphi(0)= 0 = (a_{1}W )( a_{2}'M) = \frac{1}{2}(a _{1} a_{2}')(W + M)$
and so $a_{1} a_{2}' = 0$.  Hence, $a_{1}=0$ or $a_{2}'=0$. This is a contradiction to the assumption of the claim. 

The images of $f_i$ are either both non-zero and mapped to multiples of $W$ or both non-zero and mapped to multiples of $M$. Thus, there are 18 possible maps as shown in Table \ref{tab:my_label2}.

\begin{table}[h!]
    \centering
    \caption{18 possible onto homomorphisms $\varphi$}
    \label{tab:my_label2}
    \vspace{.5cm}
        \begin{tabular}{c||ccccccccc}
    \hline \hline 
         $\varphi (f_{1})$   & $a_1'M$  &  $a_1'M$  & 0  & $a_1'M$  & $a_1'M$ & 0 & $a_1W$  & $a_1W$  & 0 \\      
         $\varphi (f_{2})$  & 0  & $a_2'M$  & $a_2'M$ & 0  & $a_2'M$ & $a_2'M$& 0 & $a_2W$  & $a_2W$ \\
        $\varphi (m_{1})$ &  $b_1W$ & $b_1W$ & $b_1W$  & $b_1W$ & $b_1W$ & $b_1W$& $b_1'M$ & $b_1'M$ &  $b_1'M$ \\
         $\varphi (m_{2})$ &  0& 0 &0  & $b_2W$  &$b_2W$ & $b_2W$& 0  & 0  & 0 \\
         \hline \hline 
         $\varphi (f_{1})$ & $a_1W$  & $a_1W$  & 0 & $a_1'M$   & $a_1'M$   & 0& $a_1W$  & $a_1W$  & 0\\
         $\varphi (f_{2})$ & 0 & $a_2W$  & $a_2W$ & 0   & $a_2'M$  & $a_2'M$ & 0 & $a_2W$  & $a_2W$\\
        $\varphi (m_{1})$ & $b_1'M$ & $b_1'M$ &  $b_1'M$ & 0 & 0 & 0  & 0 & 0 & 0  \\
         $\varphi (m_{2})$ & $b_2'M$  & $b_2'M$  & $b_2'M$ & $b_2W$ & $b_2W$ &$b_2W$ & $b_2'M$ & $b_2'M$ &$b_2'M$  \\
         \hline 
            \end{tabular}
   \end{table}
Without loss of generality, take $\varphi$ as, 
for $i,j=1,2$, $\varphi (f_{i})=a _{i}W~,~\varphi (m_{j})=b_{j}M$ where $a_i, b_j \in \mathbb{R}$.  
For 
\begin{eqnarray*}
z &=&(x,y)=x_{1}f_{1}+x_{2}f_{2}+y_{1}m_{1}+y_{2}m_{2}, \\
t &=&(u,v)=u_{1}f_{1}+u_{2}f_{2}+v_{1}m_{1}+v_{2}m_{2}
\end{eqnarray*}%
we obtain
\begin{eqnarray*}
zt
&=&u_{1}y_{1}f_{1}m_{1}+v_{1}x_{1}f_{1}m_{1}+u_{1}y_{2}f_{1}m_{2}+u_{2}y_{1}f_{2}m_{1}
\\
&&+x_{2}v_{1}f_{2}m_{1}+x_{1}v_{2}f_{1}m_{2}+u_{2}y_{2}f_{2}m_{2}+x_{2}v_{2}f_{2}m_{2}
\\
&=&(x_{1}v_{1}+u_{1}y_{1})\dfrac{1}{2}\left( f_{1}+m_{1}\right)
+(x_{1}v_{2}+u_{1}y_{2})\dfrac{1}{2}(1-w)\left( f_{1}+m_{1}\right) \\
&&+(x_{2}v_{1}+u_{2}y_{1})\dfrac{1}{2}\left[ (1-d)\left( f_{1}+m_{1}\right)
+d\left( f_{2}+m_{2}\right) \right] \\
&&+(x_{2}v_{2}+u_{2}y_{2})\dfrac{1}{2}\left[ (1-d)(1-w)\left(
f_{1}+m_{1}\right) +d\left( f_{2}+m_{2}\right) \right]
\end{eqnarray*}

\bigskip

\begin{eqnarray*}
zt &=&\dfrac{1}{2}\left\{ \left[
(x_{1}v_{1}+u_{1}y_{1})+(x_{1}v_{2}+u_{1}y_{2})(1-w)+(x_{2}v_{1}+u_{2}y_{1})(1-d)%
\begin{array}{c}
\text{ } \\ 
\text{ }%
\end{array}%
\right. 
\begin{array}{c}
\text{ } \\ 
\text{ }%
\end{array}%
\right. \\
&&\left. 
\begin{array}{c}
\text{ } \\ 
\text{ }%
\end{array}%
+(x_{2}v_{2}+u_{2}y_{2})(1-d)(1-w)\right] \left( f_{1}+m_{1}\right) \\
&&\left. 
\begin{array}{c}
\text{ } \\ 
\text{ }%
\end{array}%
+(x_{2}v_{1}+u_{2}y_{1}+x_{2}v_{2}+u_{2}y_{2})d\left( f_{2}+m_{2}\right)
\right\}
\end{eqnarray*}

\bigskip

\begin{eqnarray*}
\varphi (zt) &=&\dfrac{1}{2}\left\{
[(x_{1}v_{1}+u_{1}y_{1})+(x_{1}v_{2}+u_{1}y_{2})(1-w)+(u_{2}y_{1}+x_{2}v_{1})(1-d)%
\right.  \\
&&\left. +(u_{2}y_{2}+v_{2}x_{2})(1-d)(1-w)\right] \left( \alpha _{1}W+\beta
_{1}M\right)  \\
&&\left. +(u_{2}y_{1}+v_{1}x_{2}+u_{2}y_{2}+v_{2}x_{2})d\left( \alpha
_{2}W+\beta _{2}M\right) \right\} 
\end{eqnarray*}%
On the other hand 
\begin{eqnarray*}
\varphi (z)\varphi (t) &=&[(x_{1}a_{1}+x_{2}a_{2})W+(y_{1}b_{1}+y_{2}b_{2})M][(u_{1}a_{1}+u_{2}a_{2})W+(v_{1}b_{1}+v_{2}b_{2})M] \\
&=&\dfrac{1}{2}\{a _{1}b_{1}(x_{1}v_{1}+u_{1}y_{1})+a
_{1}b _{2}(x_{1}v_{2}+u_{1}y_{2}) \\
&&+a _{2}b _{1}(x_{2}v_{1}+u_{2}y_{1})+a _{2}b
_{2}(x_{2}v_{2}+u_{2}y_{2})\}(W+M)
\end{eqnarray*}%
Since $\varphi $ is a homomorphism $\varphi (zt)=\varphi (z)\varphi (t)$ \
and comparing coefficients of $W$ and $M$ we get

\begin{align}
a_{1}& =a _{1}b_{1}\text{ \ and \ }b_{1}=a_{1}b _{1} \label{eqn:1}\\
(1-d)a_{1}+da_{2}& =a_{2}b_{1}\text{ \ and \ }%
(1-d)  b_{1}+db_{2}=a _{2}b_{1} \label{eqn:2}\\
(1-w)a_{1}& =  a_{1}b_{2}\text{ \ and \ }(1-w)b_{1}=a_{1}b_{2} \label{eqn:3}\\
(1-d)(1-w)a_{1}+da_{2}& =a_{2}b_{2}\text{ \ and \ }%
(1-d)(1-w)b_{1}+db_{2}=a_{2}b_{2} \label{eqn:4}
\end{align}

By Equation (\ref{eqn:1}), $a_{1}=b_{1}=0$ \ or \ $a_{1}=b_{1}=1.$ 

Case 1: If $a_{1}=b_{1}=0,$ then by Equation (\ref{eqn:2}), either $d=0$ or $a_{2}=0=b_{2}$ However, if the maternal transmission rate is zero $d=0$, then the offspring population will not be infected anymore, which is not biologically meaningful. Hence, 
$a_{2}=b_{2}=0$ and $\varphi$ is the zero homomorphism. Contradiction to the assumption that $\mathcal{W}$ is dibaric. 

Case 2: If $a_{1}= b_{1}=1,$ then by Equation (\ref{eqn:2}),  
either $d=0$ or $a_2 = b_2$. As $d=0$ is not possible, take $a_2 = b_2$. By Equation (\ref{eqn:3}),
$a_{2}= b_{2}=1-w.$ Then from Equation (\ref{eqn:4}),  either $w=0$ or $w=1$. Again, $w=0$ is not considered, so $ a_{2}=b_{2}=0$. 
Now, plug the values into Equation (\ref{eqn:2}), to get $d-1 =0 $. Hence, $d=1$ also.
That is, there is a non-zero homomorphism $\varphi$ from $\mathcal{W}$ onto $\mathcal{{U}}$, such that  
$\varphi (f_1)= W$, $\varphi (m_1)= M$, $\varphi (f_2)= 0$, $\varphi (m_2)= 0$. 
(The other possible non-zero homomorphism  $\varphi$ maps the basis vectors as follows:  $\varphi (f_1)= M$, $\varphi (m_1)= W$, $\varphi (f_2)= 0$, $\varphi (m_2)= 0$.) 

Therefore, if the algebra $\mathcal{W}$ is dibaric then $w=1$ and $d=1$.

\bigskip
Conversely, assume $w=1=d$, and show that
$\mathcal{W}$, whose multiplication is given in Table \ref{tab:dibaric}, is a dibaric algebra.  

\begin{table}[h!]
\caption{The multiplication table of $\mathcal{W}$ when $w=1=d$.} 
\label{tab:dibaric} 
$
\begin{tabular}{c|c|c}
& $XY$ & $XY^{+}$ \\ \hline\hline
&   &   \\ 
$XX$ & $\frac{1}{2}XX+\frac{1}{2}XY$ & --- \\
&   &   \\ \hline
&   &   \\ 
$XX^{+}$ & $
\frac{1}{2}XX^{+} +\frac{1}{2}XY^{+}
$ & $
\begin{array}{c}
\frac{1}{2}XX^{+} +\frac{1}{2}XY^{+}
\end{array}
$
\end{tabular} $
\end{table}

Define $\varphi: \mathcal{W} \rightarrow  \mathcal{{U}}$ as $\varphi (f_1)= W$, $\varphi (m_1)= M$, $\varphi (f_2)= 0$, $\varphi (m_2)= 0$. A straightforward computation reveals that $\varphi$ is an onto homomorphism.    
\end{proof}
\subsection{$\mathcal{W}$ contains the sex differentiation algebra as a subalgebra}
As Table \ref{tab:zygote} reveals, the evolution algebra of a Wolbachia-infected bisexual population $\mathcal{W}$ contains a sex differentiation subalgebra $\mathcal{U} =  \langle  f_1, m_1 \rangle = \langle  XX, XY  \rangle$. However,  $\mathcal{U} \cdot \mathcal{W} \not \subseteq \mathcal{U}$, so the subalgebra $\mathcal{U}$ is not an ideal of $\mathcal{W}$. 

Sex differentiation algebra $\mathcal{U}$ is a dibaric algebra by definition. Moreover,  $\mathcal{U}^2 = \left\langle f_1+m_1 \right\rangle$ is a baric algebra which is an ideal of $\mathcal{U}$ and a subalgebra of $\mathcal{W}$.  

\section{Fixed points of Wolbachia-infected populations } \label{section:evolop}
This problem is also considered in the first paper of the sequel {\em An Algebraic Discussion of Bisexual Populations with Wolbachia Infection I: Discrete Dynamical System Approach,} \cite{I} via the discrete dynamical system approach. The fixed points are calculated and classified as stable and non-stable. Here, we use the evolutionary operator to arrive to the same conclusion.

\bigskip 
Consider the map $V : S^{1} \times S^{1} \rightarrow S^{1} \times  S^{1}$ where 
\begin{equation*}
S^{1} \times  S^{1} =\left\{ (x,y) \in \mathbb{R}^{2} \times \mathbb{R}^{2}~|~ (x_{1},x_{2}), (y_{1},y_{2})\in 
\mathbb{R}^{2}, \ x_{i}, y_{i} \geq 0,  \sum\limits_{i=1}^{2}x_{i}=1 =\sum\limits_{i=1}^{2}y_{i} \right\} 
\end{equation*}
with $z^{0}= (x_1^0, x_2^0, y_1^0, y_2^0) \longmapsto z^{\prime }=(x_1^{\prime }, x_2^{\prime }, y_1^{\prime }, y_2^{\prime })$.

In the evolution algebra of a bisexual population (EABP) set-up, inheritance coefficients satisfy the Equation (\ref{sumprobcoef=1}) and 
$z^{\prime }=(x_1^{\prime }, x_2^{\prime }, y_1^{\prime }, y_2^{\prime })$ is defined as 
\begin{equation*}
    x_j^{\prime } = \sum\limits_{i,k=1}^{2}P_{ik,j}^{(f)}x_iy_k  \quad   \quad
    y_j^{\prime } =\sum\limits_{i,k=1}^{2}P_{ik,j}^{(m)} x_iy_k  \quad \text{ for } j = 1,2.
\end{equation*}

However, in $\mathcal{W}$, there are deaths in the offspring, hence Equation (\ref{sumprobcoef=1}) is not satisfied. To overcome this flaw, normalize the offspring vector with respect to the $\ell^1$-norm.  
 Define  
\begin{equation}
    x_j^{\prime } = \dfrac{\sum\limits_{i,k=1}^{2}P_{ik,j}^{(f)}x_iy_k}{\sum\limits_{i,j,k=1}^{2}P_{ik,j}^{(f)}x_iy_k},   
    \quad  \quad y_j^{\prime } =\dfrac{\sum\limits_{i,k=1}^{2}P_{ik,j}^{(m)}x_iy_k}{\sum\limits_{i,j,k=1}^{2}P_{ik,j}^{(m)}x_iy_k} \quad
    \text{ for } j = 1,2.
\end{equation}  
By using Table \ref{tab:coef}, a direct computation shows that 
$$ \sum\limits_{i,j,k=1}^{2}P_{ik,j}^{(f)}x_iy_k  = 1- wy_2+ dw x_2y_2 =  \sum\limits_{i,j,k=1}^{2}P_{ik,j}^{(m)}x_iy_k .$$ 

This map $V$ is the evolutionary operator that describes the inheritance process of $\mathcal{W}$ from one generation (initial state vector) to the offspring (next generation's state vector). Namely, $V(z)= \dfrac{z \cdot z}{|| z\cdot z ||_1}$ where 
$|| z\cdot z ||_1 = 1- wy_2+dw x_2y_2$. 

Two important biological questions arise: 

Q1) For what value of $z$, $V(z) = z$? (The population stabilizes at this state vector.)

Q2) For what value of $z$, $V(z) = 0$? (The population dies if reaches this state vector.)

The element $z$ that solves Q1) is an idempotent, the element $z$ that solves Q2) is an absolute nilpotent element of the algebra $\mathcal{W}$. As both $d, w$ are parameters of $\mathcal{W}$, the solution should depend on the \emph{CI} paternal affection rate $w$ and maternal transmission rate $d$. 
\bigskip 
 
\subsection{Idempotent elements of $\mathcal{W}$}
An idempotent element of $\mathcal{W}$ is a fixed point of the operator $V$. Let $z =(x,y)=x_{1}f_{1}+x_{2}f_{2}+y_{1}m_{1}+y_{2}m_{2}$, be a fixed point of $V$. Then  $V(z)= \dfrac{z \cdot z}{|| z\cdot z ||_1} =z$, that is  
$ z \cdot z = (\sum\limits_{i,j,k=1}^{2}P_{ik,j}x_iy_k ) z =  (1- wy_2+ dw x_2y_2) z$. 
Now, 
\begin{align*}
z \cdot z  =  \left [ x_{1}y_{1}+x_{1}y_{2}(1-w)+x_{2}y_{1}(1-d) \right. &  +\left. x_{2}y_{2}(1-d)(1-w)\right ] 
\left( f_{1}+m_{1} \right)  \\ 
    & \ \ \ \ \  + \left [ (x_{2}y_{1}+x_{2}y_{2})d \right ]  \left( f_{2}+m_{2}\right) \ \ \ \ (\ast)
\end{align*}

and 
 $$  ({\sum\limits_{i,j,k=1}^{2}P_{ik,j}x_iy_k}) z = (1- wy_2+dw x_2y_2) (x_{1}f_{1}+x_{2}f_{2}+y_{1}m_{1}+y_{2}m_{2}). \ \ \ \  \  \ (\ast \ast) 
 $$

Note that the coefficients of $f_i$ and $m_i$ are equal in $z \cdot z$, so $x_{1} = y_{1}$ and $ x_{2} = y_{2}$.  

(Notice that this result agrees with the \cite[Proposition 3.1]{I}, even if the frequencies of initial Wolbachia infected male and female populations differ, in the next generation $f_2$ and $m_2$ frequencies become equal, furthermore, they are equal in the fixed point.)
 Also, use the fact that $x_1+x_2=1$ and simplify both $(\ast)$ and $( \ast \ast)$ using $x_2$ as the parameter: 
 $$  z \cdot z = \left[
(1-x_{2})^2+(1-x_{2})x_{2}(2-w-d)
+x_{2}^{2}(1-d)(1-w)\right] \left( f_{1}+m_{1}\right) +dx_{2}\left( f_{2}+m_{2}\right)  $$ 
$$(1- wx_2+dw x_2^2) z = (1- wx_2+dw x_2^2)(1-x_{2})(f_{1}+m_1)+ (1- wx_2+dw x_2^2) x_{2}(f_{2}+m_{2}).$$
 As the equations in $(\ast)$ and $( \ast \ast)$ are equal, the coefficients of the basis vectors are the same in both equations.  
 The coefficient of $f_1$ is
 $$ (1- wx_2+dw x_2^2) (1-x_2) = \left[
(1-x_{2})^2+(1-x_{2})x_{2}(2-w-d)
+x_{2}^{2}(1-d)(1-w) \right]$$ and 
the coefficient of $f_2$ is $$  (1- wx_2+dw x_2^2) x_2 =dx_{2}. $$
Simplify to get the same equation in both equalities: 
 $$ -x_{2}( dw(x_2)^2- wx_{2} + (1-d)) =0  \quad \mbox{ and } \quad dw(x_2)^2- wx_{2} + (1-d) =0. $$
 
 If $x_2 = 0$, then $x_1 =1$, hence the population is at a trivial fixed point, and there is no Wolbachia infection in the population. In this case, $ f_1+ m_1$ is an idempotent element.

 Also, there are two solutions to
 $dw(x_2)^2- wx_{2} + (1-d) =0$, namely \\ 
 $x_2= \frac{1}{2d}\left( 1+ \frac{\sqrt{w(4d^2-4d+w)}}{w} \right)$ and $x_2= \frac{1}{2d}\left( 1- \frac{\sqrt{w(4d^2-4d+w)}}{w} \right)$ provided that $0 < d(1-d) < w/4$. 

On the other hand, if $d=1$.\\

$x_2= \frac{1}{2}\left( 1 \pm  \frac{\sqrt{w(4-4+w)}}{w} \right)$ implies 
$x_2=0$ or $x_2=1$. Summarizing,

 \bigskip 
$x_2 \in \left\{ 
\begin{array}{ccc} 
\{ 0,1 \}  & \text{if} & d=1 \\
\text{ \ } \\
\left\{ \begin{array}{c}
\frac{1}{2d}\left( 1\pm \frac{\sqrt{w(4d^2-4d+w)}}{w} \right) 
\end{array} \right\}
 & \text{if} & 0 < d(1-d) < w/4
 
\end{array}
\right. $
\\

\bigskip

Hence, the idempotent element $z$ will be: 

\bigskip

Case 1. If $d=1$, then either 

a) $x_2=0$ implies $x_1=1$. Then $z=f_1 +m_1$. 

b) $x_2=1$ implies $x_1=0$. Then $z=f_2 +m_2$,

To verify,  
 a) $$ z \cdot z = (f_1 +m_1) (f_1 +m_1) =  2 \frac{1}{2d} (f_1 +m_1) = f_1 +m_1 = z$$ 
Here, $|| z\cdot z ||_1 = 1- wy_2+dw x_2y_2 = 1 -0-0 = 1$. 

b) $$z \cdot z = (f_2 +m_2) (f_2 +m_2) =  (1-w)(1-d) (f_1 +m_1) +  d (f_2 +m_2) = f_2 +m_2 = z$$ 
Here, $|| z\cdot z ||_1 = 1- wy_2+dw x_2y_2 = 1 - w +w = 1$. 

\bigskip 

Case 2. If $0 < d(1-d) < w/4$,  
\bigskip

 $x_2= \frac{1}{2d}\left( 1 \pm \frac{\sqrt{w(4d^2-4d+w)}}{w} \right)$  
  implies  $x_1= 1 -  \frac{1}{2d}\left( 1 \pm \frac{\sqrt{w(4d^2-4d+w)}}{w} \right)$. Then
$$z= \left( 1 -  \frac{1}{2d} \mp \frac{\sqrt{w(4d^2-4d+w)}}{2dw} \right) (f_1 +m_1) + \left( \frac{1}{2d} \pm \frac{\sqrt{w(4d^2-4d+w)}}{2dw} \right) (f_2 +m_2) $$

To verify:  

Here, a computation reveals that $|| z\cdot z ||_1 = 1- wy_2+dw x_2y_2 =1- wx_2+dw (x_2)^2$ 
$$  = 1-\frac{w}{2d}\left( 1 \pm \frac{\sqrt{w(4d^2-4d+w)}}{w} \right) + \frac{w}{4d}\left( 1\pm \frac{\sqrt{w(4d^2-4d+w)}}{w} \right)^2 = d. $$

$$z \cdot z = \left( \frac{2dw - w \mp \sqrt{w(4d^2-4d+w)}}{2w} \right) (f_1 +m_1) + \frac{1}{2d}\left( 1\pm  \frac{\sqrt{w(4d^2-4d+w)}}{w} \right)d(f_2 +m_2) $$

Note that,  $\dfrac{z \cdot z}{|| z\cdot z ||_1} = \dfrac{1}{d} (z \cdot z) = z $ as expected.  

\bigskip 

Summarizing the results, the set of idempotent elements of $\mathcal{W}$, denoted by $Idem(\mathcal{W})$ is given in Theorem \ref{idempotent}. 

\begin{thm} \label{idempotent} When $d=1$, $Idem(\mathcal{W})= \{ f_1+m_1, f_2+m_2\}$. \\
When  $0 < d(1-d) \leq w/4$,
$$Idem(\mathcal{W}) = \left\{ 
\begin{array}{c}
\left( 1 -  \frac{1}{2d} \mp \frac{\sqrt{w(4d^2-4d+w)}}{2dw} \right) (f_1 +m_1) + \left( \frac{1}{2d} \pm \frac{\sqrt{w(4d^2-4d+w)}}{2dw} \right) (f_2 +m_2)  
\end{array}
\right\}  $$
\end{thm}

\bigskip

Theorem \ref{idempotent} agrees with the fixed points calculated via the dynamical system approach in \cite[Section 3]{I}. For instance, 
\begin{enumerate}
    \item 
When $d=w=\frac{3}{4}$, the fixed point is $z=\frac{1}{3}(f_1+m_1) + \frac{2}{3} (f_2 +m_2)$ \cite[Case 2 of Section 3]{I}. 

\item 
When $w=1$, that is CI expression is 100\%, either, $$x_2= \frac{1}{2d}+ \frac{\sqrt{(4d^2-4d+1)}}{2d}= \frac{1 + 2d-1}{2d} = 1$$ 
The idempotent element is 
$z =0 (f_1+m_1) + 1 (f_2+m_2)$.

or the idempotent depends on $d \in (0,1]$ with 
$$x_2= \frac{1}{2d}- \frac{\sqrt{(4d^2-4d+1)}}{2d}= \frac{1 - 2d+1}{2d}=\frac{1 - d}{d}$$ 
provided that $\frac{1 - d}{d} \leq 1$, (i.e. $\frac{1}{2} \leq d$). The idempotent element is 
$z = \frac{2d-1}{d} (f_1+m_1) + \frac{1-d}{d} (f_2+m_2)$. \\

Also, take $w=1, d=1/2$. Then $z=f_2+m_2$ is the only idempotent element. 

Hence, for $w=1$, if the maternal transmission rate $d$ is greater than 50 \% , there are two fixed points of the population. Otherwise, when $0< d \leq \frac{1}{2}$, the only idempotent is $(f_2+m_2)$. \\

Consider a particular scenario: take a population with $w=1, d=\frac{2}{3}$. The inequality $d(1 - d)=\frac{2}{9} \leq \frac{w}{4}= \frac{1}{4}$ is satisfied and  $x_2= \frac{1 - d}{d} = \frac{1}{2} \leq 1$. Then $z=\frac{1}{2} (f_1+m_1) + \frac{1}{2} (f_2+m_2)$ is an idempotent.   
 
\end{enumerate}

\bigskip

\subsection{Absolute nilpotent elements of $\mathcal{W}$}
The element $z$ is called \textit{absolute nilpotent} if $z^2=0$. \\ Let $z =(x,y)=x_{1}f_{1}+x_{2}f_{2}+y_{1}m_{1}+y_{2}m_{2}$ be an absolute nilpotent element of $\mathcal{W}$, that is $V(z)= \dfrac{1}{|| z ||_1} (z \cdot z) = 0$ where 
$ || z ||_1 = \sum\limits_{i,j,k=1}^{2}P_{ik,j}x_iy_k  =  1- wy_2+dw x_2y_2 $.   

Hence, $$ 0 = z \cdot z = 
\left[ x_{1}y_{1}+x_{1}y_{2}(1-w)+x_{2}y_{1}(1-d)+x_{2}y_{2}(1-d)(1-w)\right]
\left( f_{1}+m_{1}\right) $$ 
$$+(x_{2}y_{1}+x_{2}y_{2})d\left( f_{2}+m_{2}\right).$$ 
Consequently, 
 $$ (I)  \quad x_{1}y_{1}+x_{1}y_{2}(1-w)+x_{2}y_{1}(1-d)
+x_{2}y_{2}(1-d)(1-w) =0 $$  and $$ (II) \quad (x_{2}y_{1}+x_{2}y_{2})d =0.$$
Take $(II)$, since $y_{1}+y_{2}=1$, and $d\neq 0$, $x_{2}=0$, and  $x_{1}=1$. 
Substituting $x_{2}=0$ and $x_{1}=1$ in $(I)$, gives $y_{1}+y_{2}(1-w)=0$. 
Then, $ 1 -wy_2=0$ implies $y_{2}=\dfrac{1}{w}$.
Note that $ w \leq 1$ implies $1/w \geq 1$ and $y_i \leq 1$ for $i=1,2$. Thus, $y_{2}=\dfrac{1}{w} = 1$ occurs when $w=1$. Consequently, $y_{1}=0$, and $z=f_1+m_2$ is an absolute nilpotent element of $\mathcal{W}$ with $w=1$. The set of all absolute nilpotent elements of $\mathcal{W}$ with $w=1$ will be the subalgebra generated by $\langle f_1 + m_2 \rangle$. 
If $w \neq 1$, there are no non-zero absolute nilpotent elements, i.e. the following theorem is proved: 
\begin{thm} \label{nilpotent}
For $\mathcal{W}$ with $\mathcal{CI}$ given as $w$ and maternal transmission rate $d$, the set of absolute nilpotent elements $Nil(\mathcal{W})$ is zero except when $w=1$. That is,  for any value of $d \in (0,1]$,  
$$Nil(\mathcal{W})= \left\{ 
\begin{array}{ccc}
\langle f_1 + m_2 \rangle  & \text{if} & w=1 \\ 
\text{ \ } &  &  \\ 
  \{ 0\}  & \text{if} & w\neq 1.
\end{array}
\right.$$ 
\end{thm}

\end{document}